\documentclass[a4paper,12pt]{article}
\DeclareMathAlphabet{\mathfr}{U}{euf}{m}{n}
\usepackage{latexsym}
\usepackage{amsmath}
\usepackage{amssymb}
\usepackage{amscd}
\usepackage{array}
\usepackage{comment}
\usepackage[all]{xy}
\usepackage{amsthm}
\usepackage{amsfonts}

\newtheorem{thm}{Theorem}[section]

\newtheorem{rem}{Remark}[section]

\newtheorem{lema}{Lemma}[section]
\newtheorem{coro}{Corollary}[section]
\newtheorem{prop}{Proposition}[section]
\newtheorem{example}{Example}[section]

\newcommand{\Q}{\mathbb Q}

\newcommand{\Z}{\mathbf Z}

\newcommand{\Gal}{\mathrm{Gal}}

\newcommand{\Hol}{\mathrm{Hol}}

\newcommand{\Sym}{\operatorname{Sym}}
\newcommand{\disc}{\operatorname{disc}}

\newcommand{\End}{\operatorname{End}}

\newcommand{\Aut}{\operatorname{Aut}}

\newcommand{\Id}{\operatorname{Id}}

\newcommand{\wk} {{\widetilde{K}}}

\title{The Hopf Galois property in subfield lattices}
\author{Teresa Crespo, Anna Rio and Montserrat Vela}
\date{\today}

\begin{document}
\maketitle

\let\thefootnote\relax\footnote{T. Crespo acknowledges support by grant MTM2012-33830, Spanish Science Ministry, and 2009SGR 1370; A. Rio and M. Vela acknowledge support by grant MTM2012-34611, Spanish Science Ministry, and 2009SGR 1220.}

\noindent {\bf Abstract.} Let $K/k$ be a finite separable
extension, $n$ its degree and $\wk/k$ its Galois closure. For
$n\le 5$, Greither and Pareigis show that all Hopf Galois
extensions are either Galois or almost classically Galois and they
determine the Hopf Galois character of $K/k$ according to the
Galois group (or the degree) of $\wk/k$. In this paper we study
the case $n=6$, and intermediate extensions $F/k$ such that
$K\subset F\subset \wk$, for degrees $n=4,5,6$. We present an
example of a non almost classically Galois Hopf Galois extension
of $\Q$ of the smallest possible degree and new examples of Hopf
Galois extensions. In the last section we prove a transitivity
property of the Hopf Galois condition.

\vspace{0.2cm} \noindent {\bf Keywords.} Hopf algebra, Hopf Galois
extension, holomorph.

\vspace{0.2cm} \noindent {\bf MSC2010.}  16T05, 12F10, 11R32.

\section{Introduction}

Following  \cite{Childs}, if $K/k$ is a finite extension of
fields, we say that $K/k$ is a Hopf Galois extension if there
exists a finite cocommutative $k-$Hopf algebra   $H$  such that
$K$ is an $H-$module algebra and the $k-$linear map $j:K\otimes_k
H\to\End_k(K)$, defined by $j(s\otimes h)(t)=s(ht)$ for $h\in H$,
$s,t\in K$, is bijective. The following \emph{main theorem} holds.

\begin{thm}[\cite{ChS}]
Let $K/k$ be a Hopf Galois extension with algebra $H$ and  Hopf action $\mu: H\to \End(K)$. For a $k$-sub-Hopf algebra $H'$ of $H$ we define
$$
K^{H'}=\{x\in K \mid \mu(h)(x)=\epsilon(h)\cdot x \mbox{ for all } h\in H'\}
$$
where $\epsilon$ is the counit of $H$. Then, $K^{H'}$ is a
subfield of $K$, containing $k$, and the correspondence
$$
\begin{array}{rcl}
{\cal F}_H:\{H'\subseteq H \mbox{ sub-Hopf algebra}\}&\longrightarrow&\{\mbox{Fields }E\mid k\subseteq E\subseteq K\}\\
H'&\to &K^{H'}
\end{array}
$$
is injective and inclusion reversing.
\end{thm}

\vspace{0.2cm} \noindent {\bf Notation.} Throughout this paper,
$k$ will denote a field, $K/k$ a separable extension of degree
$n$, $\wk/k$  its Galois closure, $G$ the Galois group of $\wk/k$
and $G'$ the Galois group of $\wk/K$.

\vspace{0.2cm} For separable field extensions, Greither and
Pareigis \cite{GP} give the following group-theoretic
characterization of the Hopf Galois property.

\begin{thm}
$K/k$ is a Hopf Galois extension if and only if there exists a
regular subgroup $N$ of $S_n$ normalized by $\lambda (G)$, where
$\lambda:G \rightarrow S_n$ is the morphism given by the action of
$G$ on the left cosets $G/G'$, i.e.

$$
 \begin{array}{rl}
\lambda: & G\rightarrow \Sym(G/G') \simeq S_n\\
 &g\mapsto (\lambda_g: xG'\mapsto gxG')\, .
\end{array}
$$
\end{thm}

If there exists a regular subgroup $N$ of $S_n$ normalized by
$\lambda(G)$ and contained in $\lambda(G)$, we say that $K/k$ is
an \emph{almost classically Galois extension}.

\begin{thm} (\cite{GP})
$K/k$ is almost classically Galois if and only if $G'$ has a
normal complement $N$ in $G$.
\end{thm}

In particular, if $K/k$ is Galois, then $G'=1$ and has normal
complement $N=G$. The following theorem provides a justification
for the notion of almost classically Galois extensions.

\begin{thm}[\cite{GP}]
If $K/k$ is almost classically Galois, then there is a Hopf
algebra $H$ such that $K/k$ is Hopf Galois with algebra $H$ and
the main theorem holds in its strong form, i.e. the correspondence
${\cal F}_H$ between $k$-sub-Hopf algebras of $H$ and
$k$-subfields of $K$ is bijective.
\end{thm}

It is known that all Hopf Galois extensions of degree $n \le 7$
are almost classically Galois extensions.

To perform our computations we use the following reformulation due
to Byott \cite{Byott96}.

\begin{thm} Let $G$ be a finite group, $G'\subset G$ a subgroup and $\lambda:G\to \Sym(G/G')$ the morphism defined above.
Let $N$ be a group of
order $[G:G']$ with identity element $e_N$. Then there is a
bijection between
$$
{\cal N}=\{\alpha:N\hookrightarrow \Sym(G/G') \mbox{ such that
}\alpha (N)\mbox{ is regular}\}
$$
and
$$
{\cal G}=\{\beta:G\hookrightarrow \Sym(N) \mbox{ such that }\beta
(G')\mbox{ is the stabilizer of } e_N\}
$$
Under this bijection, if $\alpha\in {\cal N}$ corresponds to
$\beta\in {\cal G}$, then $\alpha(N)$ is normalized by
$\lambda(G)$ if and only if $\beta(G)$ is contained in the
holomorph $\Hol(N)$ of $N$.
\end{thm}

We consider Byott reformulation as an ``algorithmic" procedure to
check if a given extension $K/k$ is Hopf Galois.
\begin{description}
\item[Step 1:] Let $N$ run through a system of representatives of isomorphism classes of groups of order
$n$;
\item[Step 2:] Compute $\Hol(N)\subseteq S_n$;
\item[Step 3:] Check if $G\subseteq \Hol(N)$.
\end{description}
In the third step, we should be aware that $S_n$ may have
different conjugacy classes of transitive subgroups isomorphic to
$G$, namely that the embedding $G\hookrightarrow \Sym(G/G')\simeq
S_n$ must be taken into account.

In this paper, we shall follow this algorithmic procedure with
$n=4,\ 5,\ 6$ to describe the Hopf Galois character of separable
extensions $K/k$ and extensions $F/k$ for  fields $F$ such that
$K\subset F\subset \wk$.  Although most of the work has to be done
case by case, we include here some useful generic results.

\begin{lema}
Let us consider  a dihedral group
$$D_{2n}=\langle s,r| s^2=1,\ r^n=1,\ sr=r^{-1}s\rangle$$
and a subgroup $G'$ of order $2$. If $G'$ is not normal, then the
cyclic subgroup  $N=\langle r\rangle$ is a normal complement of
$G'$.
\end{lema}

\begin{lema}
Let $F$ be a Frobenius group. If $N$ is the Frobenius kernel and
$G'$ is the Frobenius complement, then $N$ is a normal complement
of $G'$ in $F$.
\end{lema}

\begin{lema}
If $n\ge 3$, a subgroup of order 2 of the symmetric group $S_n$ is
never normal. It has a normal complement if and only if its
nontrivial element is an odd permutation. In that case, the normal
complement is the alternating group $A_n$.
\end{lema}


\section{Hopf Galois in degree 4}

In the case $n=4$, the character of $K/k$, for each possible
Galois group $G$, is given in \cite{GP}, theorem 4.6. We list the
results in the following table

\begin{center}
\begin{tabular}{l | c | l }
$G$ & $|G|$  & $K/k$\\
\hline
\hline
$C_4$ & 4  & Galois\\
$C_2\times C_2$ & 4  &  Galois\\
$D_{2\cdot 4}$ & 8  & almost classically Galois\\
$A_4$  &12  & almost classically Galois\\
$S_4$  &24  & almost classically Galois\\
\end{tabular}
\end{center}

Now we are interested in intermediate fields $K\subset F \subset
\wk$ and the Hopf Galois condition  for $F/k$. We only have
nontrivial intermediate fields when $\Gal(\wk/k)=S_4$ and then
$G'=\Gal(\wk/K)$ is isomorphic to $S_3$. Since its order 2
elements are transpositions, subgroups of order 2 of $G'$ have
normal complement $A_4$ and the corresponding extensions are
almost classically Galois.

Let us consider now the order 3 subgroup $G''$ of $G'$ and the
fixed field $F=\wk^{G''}$. Since $S_4$ has no normal subgroups of
order 8,  $G''$ has no normal complement in $S_4$ and $F/k$ is not
almost classically Galois. Let us see if it is Hopf Galois by
looking for a  group $N$ of order 8 such that $G\subseteq
\Hol(N)$. Here, $G$ is identified with a transitive subgroup of
$S_8$ through the action by left multiplication on the set of left
cosets $G/G''$. In fact, any transitive subgroup of $S_8$
isomorphic to $S_4$ will be a conjugate of $G$, since $S_4$ has a
unique conjugacy class of elements (and subgroups) of order $3$.
Therefore, in order to have the Hopf Galois condition, it is
enough to see that for some $N$, the holomorph $\Hol(N)$ has a
transitive subgroup $G_1$ isomorphic to $S_4$.

We have five possible abstract groups $N$, namely the abelian groups  $C_8$, $C_2\times C_4$, $C_2\times C_2\times C_2$, the dihedral group $D_{2\cdot 4}$
and the quaternion group $H_8$. If we look for holomorphs having order divisible by 24, we are left with $C_2\times C_2\times C_2$ and $H_8$.

If we take $N=H_8$, then $\Hol(N)$ has no transitive subgroups isomorphic to $S_4$. However, for $N=C_2\times C_2\times C_2$ there are
such subgroups. For example, if
$$
N=\langle(1, 6)(2, 7)(3, 5)(4, 8),  (1, 4)(2, 3)(5, 7)(6, 8),(1, 3)(2, 4)(5, 6)(7, 8)\rangle, $$
 then its normalizer in $S_8$ is

 $$\Hol(N)= \langle  N,(2, 3)(5, 7),
    (2, 7)(3, 5),  (2, 4)(7, 8), (2, 6)(3, 5, 8, 4)\rangle,$$
which has a transitive subgroup isomorphic to $S_4$, namely $$
G_1=\langle(1, 7, 4, 2)(3, 6, 5, 8),(1, 2)(3, 7)(4, 6)(5,
8)\rangle. $$ Therefore $F/k$ is a Hopf Galois extension. We have
obtained the following result.

\begin{prop}
Let $K/k$ be a separable extension of degree 4 with Galois closure
$\wk$. If $\Gal(\wk/k)\simeq S_4$  and $F$ is a  field with
 $K\subset F\subset \wk$, then $F/k$ is Hopf Galois. For $[F:k]=12$ (resp. $[F:k]=8$),
 it is (resp. is not) almost classically
 Galois.
\end{prop}

\begin{coro}\label{cor} Let $f \in \Q[X]$ be an irreducible polynomial of
degree 4, with Galois group $S_4$ and $x$ be a root of $f$ in a
splitting field of $f$. Then the extension $\Q(x,\sqrt{d})/\Q$,
where $d$ denotes the discriminant of $f$,  is a degree 8 Hopf
Galois extension which is not almost classically Galois.
\end{coro}

\begin{example}{\rm We give now an explicit example of a degree 8 extension  $F/\Q$ as in Corollary \ref{cor}. We consider
the irreducible polynomial $f=X^4+X+1\in\Q[X]$ and denote by $\wk$
its splitting field. We have $\Gal(\wk/\Q)\simeq S_4$ and
$\disc(f)=229$. Then $F=\Q(x, \sqrt{229})$, for $x$ a root of $f$
in $\wk$. By using Cardano's formulas, and making a choice for
$x$, we obtain

$$F=\Q(\sqrt{u-v}+\sqrt{\omega u-\omega^2 v}+\sqrt{\omega^2u-\omega
v}\, , \sqrt{229}\,), $$ where

$$u=\sqrt[3]{\frac12+\frac16 \sqrt{\frac{-229}{3}}} \qquad , \qquad
v= \omega \sqrt[3]{\frac{-1}{2}+\frac16 \sqrt{\frac{-229}{3}}}$$
and $\omega$ is a primitive cubic root of unity.

Let us note that the example of a Hopf Galois non almost
classically Galois extension given in \cite{GP} is a degree 16
extension of a quadratic number field. }
\end{example}

\section{Hopf Galois in degree 5}

In the case $n=5$, the character of $K/k$, for each possible
Galois group $G$, is given in \cite{GP}, theorem 4.6. We list the
results in the following table

\begin{center}
\begin{tabular}{l | c | l }
$G$ & $|G|$ & $K/k$\\
\hline
\hline
$C_5$ &5  & Galois\\
$D_{2\cdot 5}$ & 10& almost classically Galois\\
$F_5$ & 20& almost classically Galois\\
$A_5$ &60 & not Hopf Galois\\
$S_5$ &120 & not Hopf Galois\\
\end{tabular}
\end{center}

We note that in \cite{GP} the possibility is included of a Galois
group of order 15 which  in fact does not occur.

Now we are interested in cases where $K/k$ is Hopf Galois and we
consider intermediate fields $K\subset F \subset \wk$ in order to
determine if $F/k$ is a Hopf Galois extension. We will only have
non trivial intermediate fields when $G$ is isomorphic to the
Frobenius group $F_5$. Since in this case we have $[\wk:K]=4$, the
extensions $F/k$ under consideration have degree 10 and then
$\Gal(\wk/F)$ has order 2. The Frobenius group $F_5$ has one
(normal) subgroup of order 10 which contains all the order 2
elements of $F_5$. Therefore, $\Gal(\wk/F)$ has no normal
complement in $G$. Hence, $F/k$ is not almost classically Galois.

We are again in the situation when $G$ has a unique conjugation
class of subgroups isomorphic to $C_2$ and $S_{10}$ has a unique
conjugation class of transitive subgroups isomorphic to $G$.
Therefore, in order to check the Hopf Galois condition for $F/k$,
it suffices to find a regular subgroup $N\subset S_{10}$ such that
$\Hol(N)$ has a transitive subgroup isomorphic to $F_5$.

Since $N$ must have order 10, it can be $N\simeq C_{10}$, in which
case, $\Hol(N)$ has order 40, or $N\simeq D_{2\cdot 5}$, and then
$\Hol(N)$ has order 200. Let us take
$$N_1=\langle (1,2,3,4,5,6,7,8,9,10) \rangle\subset S_{10}.$$
Its normalizer is $ \Hol(N_1)=\langle  N_1,
    (2, 4, 10, 8)(3, 7, 9, 5)\rangle
$ whose transitive subgroup $\langle(1, 2, 5, 4)(3, 8)(6, 7, 10,
9), (1, 3, 5, 7, 9)(2, 4, 6, 8, 10)\rangle$ is isomorphic
to~$F_5$.

If we take the regular group
$$N_2=\langle (1, 7)(2, 8)(3, 4)(5, 9)(6, 10),(1, 6, 4, 9, 8)(2, 5, 3, 10, 7)\rangle\subset S_{10},$$
which is isomorphic to  $D_{2\cdot 5}$, and its normalizer
$$
\Hol(N_2)=\langle (1, 7)(2, 8)(3, 4)(5, 9)(6, 10),
    (2, 5, 3, 10, 7),
    (2, 3, 5, 7)(4, 8, 9, 6)  \rangle
$$
we find the subgroup $\langle (1, 7, 8, 5)(2, 4)(3, 9, 10, 6), (1, 4, 8, 6, 9)(2, 10, 5, 7, 3)\rangle$
which is also transitive and isomorphic to $F_5$.

\begin{prop}
Let $K/k$ be a separable extension of degree 5 and $\wk/k$ its
Galois closure. If $\Gal(\wk/k)$ has order 20 and $F$ is a field
with
 $K\varsubsetneq F\varsubsetneq \wk$, then the degree 10 extension $F/k$ is Hopf Galois but not almost classically Galois.
This extension has at least two different Hopf Galois structures.
\end{prop}

We consider now the cases in which $K/k$ is not Hopf Galois and we
want to compute the smallest degree $[F:k]$ for fields $F$ such
that $K\subset F \subseteq \wk$ and $F/k$ is Hopf Galois.

\begin{prop}
Let $K/k$ be a separable extension of degree 5 which is not Hopf
Galois, $\wk/k$ its Galois closure. The smallest degree $[F:k]$
for fields $F$ such that $K\subset F \subseteq \wk$ and $F/k$ is
Hopf Galois is $60$. More precisely, this smallest degree is
attained for $F=\wk$ in the case $G=A_5$ and for an almost
classically Galois extension $F/k$ in the case $G=S_5$ .
\end{prop}

\begin{proof} We consider separately the two possibilities for the Galois
group of~$\wk/k$.

\noindent {\bf First case: $\Gal(\wk/k)=A_5$.}

In this case we have $G'=\Gal(\wk/K)\simeq A_4$ 
and the nontrivial
possibilities for $\Gal(\wk/F)$ are subgroups isomorphic to   $V_4, C_3,C_2$.
We will have degrees
$
m=[F:k]=15,20,30,
$
respectively. In order to have the  Hopf Galois condition there should be a regular subgroup $N\subset S_m$ such that
$\Hol(N)$ has a subgroup isomorphic to $A_5$.

Although some cases of degree 20 could be ruled out just because
the order of the holomorph is not divisible by 60, we are going to
show that all groups of order $15,20$ or $30$ have a solvable
holomorph. Since all groups of these orders are solvable, it is
enough to check that the automorphism groups are solvable.

There is only one group of order 15 (modulo isomorphism), the cyclic one. Its automorphism group has order $\varphi(15)=8$ and is solvable.

Regarding the five groups of order 20, we have $\Aut(C_{20})\simeq
C_8$, $\Aut(C_5\times C_2\times C_2)\simeq C_4\times S_3$,
$\Aut(F_5)\simeq \operatorname{Inn}(F_5)\simeq F_5$ and
$|\Aut(D_{2\cdot 10})|=10\varphi(10)=40$. The remaining group $
N=\langle a,b\mid a^5=b^4=1, \ aba=b\rangle $ has also an
automorphism group of order $40$, since sending $a$ to one of the
four elements of order $5$ and $b$ to one of the ten elements of
order 4 uniquely determines an automorphism of $N$. Therefore, all
groups of order 20 have a solvable automorphism group (of order
$\le 40$).

Finally, we have four different groups of order 30: $C_{30},\
S_3\times C_5, D_{2\cdot 5}\times C_3$ and $D_{2\cdot 15}$. We
have $|\Aut(C_{30})|= \varphi(30)=8$ and $|\Aut(D_{2\cdot
15})|=15\varphi(15)=120$.  In fact,
$\Aut(D_{15})=\operatorname{Aff}(\Z/15\Z)$. On the other hand,
$\Aut(S_3\times C_5)=S_{3}\times C_4$ and $|\Aut(D_{2\cdot
5}\times C_3)|=40$, with $\operatorname{Inn}(D_{2\cdot 5}\times
C_3)\simeq D_{2\cdot 5}$. In all cases, the automorphism group is
solvable.

Since there are no proper intermediate extensions satisfying the
Hopf Galois condition, the minimal Hopf Galois extension we were
looking for is the Galois closure itself.

\vspace{0.2cm} \noindent {\bf Second case $\Gal(\wk/k)=S_5$.}

Now $G'=\Gal(\wk/K)\simeq S_4$ and the nontrivial possibilities
for the subgroup $\Gal(\wk/F)$ are $A_4,D_8,S_3,V_4, C_4,C_3,C_2$.
We have $ m=[F:k]=10,15,20,30,30,40,60, $ respectively, and
checking the Hopf Galois condition leads us to the search of
regular subgroups $N\subseteq S_m$ such that $\Hol(N)$ contains a
subgroup isomorphic to $S_5$. The possibilities $m=10,15,20,30$
are ruled out since we have already seen that the groups of these
orders have a solvable holomorph. We are left with groups $N$ of
order 40 or 60.

According to the databases of small groups, there are 14 isomorphism classes of groups of order 40. They are
named in Magma as $\mathtt{SmallGroup(40,i)}$ for $1\le i\le 14$. There is only one having a non solvable holomorph: it corresponds to $i=14$,
which gives the group
$C_2\times C_2\times C_2\times C_5$. The nonsolvability of the holomorph is due to its subgroup $\Aut(C_2\times C_2\times C_2)$, a simple group
of order $168$. But this holomorph does not have subgroups isomorphic to $A_5$.

On the other hand, there are $13$ isomorphism classes of groups of order $60$: $A_5$, $A_4\times C_5$, six non-isomorphic semidirect products $C_{15}\rtimes C_4$ and
five non-isomorphic semidirect products $C_{15}\rtimes V_4$. The only one having a non solvable holomorph is $A_5$, where $\Hol(A_5)=A_5\rtimes \Aut(A_5)\simeq A_5\rtimes S_5$. In degree 60 we can have a Hopf Galois extension $F/k$ attached to a regular subgroup $N\simeq A_5$.

In fact, since $G'\subset S_5$ is the stabilizer of a point,  there exists
a subgroup $C$ of $\Gal(\wk/K)$ generated by a transposition, and  we can take the fixed field $F=\wk^{C}$.
Then, $\Gal(\wk/F)$ has a normal complement $N\simeq A_5$ in $\Gal(\wk/k)$.
In other words, the extension $F/k$ is almost classically Galois.
\end{proof}

\section{Hopf Galois in degree 6}

In this section we consider a separable field extension  $K/k$ of
degree $6$. The possible groups $G=\Gal(\wk/k)$ are the transitive
permutation groups of degree 6, which are listed in the following
table. We have kept the essential information given in the naming
scheme developed in \cite{CHM}. We may also refer to these groups
as in Magma language, namely {\tt TransitiveGroup(6, i)}, or $6Ti$
for short.

We have collected here the results on the Hopf Galois condition
which are proved in this section. We will be interested in
intermediate fields within $K$ and $\wk$, so we have separated the
first five cases, where there are no proper intermediate fields.

\begin{center}

\medskip
\begin{tabular}{l|l | c | l }
& $G$ & $|G|$ & $K/k$\\
\hline
\hline
$6T1$&$C_6$&6  & Galois\\
$6T2$&$S_3$&6 & Galois\\
$6T3$&$D_{2\cdot 6}$ &12 & almost classically Galois\\
$6T4$&$A_4$ &12 & not Hopf Galois \\
$6T5$&$F_{18}$&18 & almost classically Galois\\
\hline
$6T6$&$2A_4$ &24 & not Hopf Galois\\
$6T7$&$S_4(6d )$ &24 & not Hopf Galois\\
$6T8$&$S_4(6c)$ &24 & not Hopf Galois\\
$6T9$&$F_{18} : 2$&36 & almost classically Galois\\
$6T10$&$F_{36}$ &36 & not Hopf Galois\\
$6T11$&$2S_4$ &48 & not Hopf Galois\\
$6T12$&$A_5$ &60 & not Hopf Galois\\
$6T13$&$F_{36}: 2$&72 & not Hopf Galois\\
$6T14$&$S_5$ &120 & not Hopf Galois\\
$6T15$&$A_6$ &360 & not Hopf Galois\\
$6T16$&$S_6$ &720 & not Hopf Galois\\
\end{tabular}
\end{center}

For an extension $K/k$ of degree 6, a Hopf Galois structure comes
from a group $N$ of order 6. By considering the holomorphs of
these groups some cases become very easy to decide.

\begin{prop} Let $K/k$ be a separable extension of degree 6 and $\wk/k$ its
Galois closure.
\begin{enumerate}
\item If $[\wk:k]=24$ or $[\wk:k]>36$, then $K/k$ is not Hopf Galois.
\item If $\Gal(\wk/k)\simeq A_4$, then $K/k$ is not Hopf Galois
\end{enumerate}
\end{prop}

\begin{proof} If $N$ is a group of order 6, it can be either cyclic or isomorphic to the symmetric group
$S_3$. Since $\Hol(C_6)\simeq D_{2\cdot 6}$ and $\Hol(S_3)\simeq S_3\times S_3$, none of them can have a subgroup
of order $24$ o bigger than $36$.

On the other hand, $\Hol(C_6)$ is not isomorphic to $A_4$ and for
an extension with Galois group isomorphic to $A_4$ we must look
inside $\Hol(S_3)$. If we assume that $G$ is a subgroup of order
12 of $S_3\times S_3$, then one of the projections of $G$ on the
components  has to be the whole $S_3$. Therefore $G$ has a
subgroup of order 6 and this proves $G\not\simeq A_4$.
\end{proof}

\begin{prop} Let $K/k$ be a separable extension of degree 6 and $\wk/k$ its
Galois closure. If $\Gal(\wk/k)$ is isomorphic to the dihedral
group $D_{2\cdot 6}$ or $F_{18}$ or $F_{18} : 2$, then $K/k$ is
almost classically Galois.
\end{prop}
\begin{proof}
If $G=\Gal(\wk/k)\simeq D_{2\cdot 6}$, then $G'=\Gal(\wk/K)$ is a
non normal subgroup of order 2 of $G$. Therefore, $G'$ has normal
complement the cyclic subgroup of $G$ of order $6$.

If $G=\Gal(\wk/k)\simeq F_{18}\simeq S_3\times C_3$, then
$G'=\Gal(\wk/K)$ is a non normal subgroup of order $3$ of $G$. The
group $G$ has a normal subgroup $N$ of order $6$ (isomorphic to
$S_3$) and the unique subgroup of $N$ of order 3 is also normal in
$G$. Therefore $G'\cap N=1$ and $N$ is a normal complement for
$G'$.

If $G=\Gal(\wk/k)\simeq F_{18}:2\simeq S_3\times S_3\simeq
\Hol(S_3)$,  consider $G$ as a transitive subgroup of $S_6$ and
$G'=\Gal(\wk/K)$ as the stabilizer of  a point. The two normal
subgroups of $G$ of order $6$ are $\langle (12)(36)(45), \
(153)(264)\rangle$ and $\langle(14)(25)(36), \ (135)(264)\rangle$.
Both of them are normal complements for an stabilizer $G'$.
\end{proof}

\begin{prop} Let $K/k$ be a separable extension of degree 6 and $\wk/k$ its
Galois closure. If $\Gal(\wk/k)=F_{36}$, then $K/k$ is not Hopf
Galois.
\end{prop}

\begin{proof}The group $G=\Gal(\wk/k)=F_{36}$ has order 36 but it is not isomorphic to $\Hol(S_3)$.
Since $\Hol(C_6)$ is too small, there is no regular subgroup $N$
of $S_6$ such that $\Hol(N)$ contains $G$.
\end{proof}

\subsection{Intermediate extensions in the case $G=F_{18}:2$.}

We consider now fields $F$ with $K\subset F \subset \wk$ and
analyze the Hopf Galois property for the extension $F/k$ in the
unique case in which $K/k$ is Hopf Galois of degree 6 and there
are proper intermediate fields, namely when $G=F_{18}:2$.
Therefore, $G'\simeq S_3$ and we can have $\Gal(\wk/F)\simeq C_2$
or $\Gal(\wk/F)\simeq C_3$.

\begin{prop} Let $K/k$ be a separable extension of degree 6 and $\wk/k$ its
Galois closure. Let us assume $\Gal(\wk/k)=F_{18}:2$ and let $F$
be a field with $K\subseteq F\subseteq \wk$. If $[F:k]=18$, then
$F/k$ is almost classically Galois. If $[F:k]=12$, then $F/k$ is
Hopf Galois  but not almost classically Galois.
\end{prop}

\begin{proof} To analyze the case $\Gal(\wk/F)\simeq C_2$, we give first a  more explicit description of
$G=F_{18}:2$, which is a group isomorphic to $S_3 \times S_3$. Let
$$
\begin{array}{rcl}
H&=&\langle r,s| r^3=s^2=1, rs=sr^2\rangle \\
K&=&\langle x,y| x^3=y^2=1, xy=yx^2\rangle
\end{array}
$$
be direct factors of $G$. Then, modulo conjugation, $G'=\langle
rs,xy\rangle$, since $G'$ has order 6, is a not a normal subgroup
of $G$ and does not contain the normal subgroups $\langle
r\rangle$ and $\langle x\rangle$. Then, $\Gal(\wk/F)\simeq C_2$ is
$\langle sy\rangle$, $\langle rxsy\rangle$ or $\langle
r^2x^2sy\rangle$. All of them have normal complement $N=H\times
\langle x\rangle$ (and also $N=K\times \langle r\rangle$). This
shows that $F/k$ is almost classically Galois.

Next we consider the case $\Gal(\wk/F)\simeq C_3$. Now the
extension $F/k$ is not almost classically Galois since $G$ has no
normal subgroups of order 12. We check if there is a regular
subgroup $N \subset S_{12}$ such that $G\subseteq \Hol(N)$. Now we
are considering $G$ as a transitive subgroup of degree 12,  and
transitive subgroups of $S_{12}$ isomorphic to $S_3\times S_3$
happen to be in a unique conjugacy class, which is the one denoted
by 12T16 according to the notation in \cite{CHM}. Therefore, we
are looking for a group $N$ of order 12 such that its holomorph
has a transitive subgroup isomorphic to $S_3\times S_3$.

The  cyclic group $C_{12}$ is ruled out because its holomorph has
order $12\varphi(12)$ \linebreak $=48$. The holomorph of the
alternating group$A_4$ has order  $12 \cdot 24=288$ and the
remaining three groups of order 12 have holomorphs of order 144.
Even more, the dihedral group $D_{2\cdot 6}$ and the dicyclic
group $\operatorname{Dic}_3$ have isomorphic holomorphs: the
transitive group in the class 12T81, again in the notation of
\cite{CHM}. Therefore, $\Hol(D_{2\cdot 6})$ has generators
$$
\begin{array}{ll}
g_1=(2,6,10)(4,8,12),
&g_2=(1,5)(2,10)(4,8)(7,11),\\
g_3=(1,4,7,10)(2,5,8,11)(3,6,9,12),&
g_4=(1,7)(3,9)(5,11).
\end{array}
$$
The subgroup $\langle g_1^2, \ g_4g_3, \ g_1g_3g_1^{-1}g_3^{-1},\
g_2g_3^2\rangle$ is transitive and isomorphic to $S_3\times S_3$
(hence to $F_{18}:2$).
\end{proof}



Let us remark that we have seen that the extension $F/k$ with
$[F:k]=18$ has at least two different Hopf Galois structures. The
corresponding Hopf algebras are obtained by Galois descent of
$\wk[D_{2\cdot 6}]$ and $\wk[\operatorname{Dic}_3]$.

\subsection{Intermediate extensions for $K/k$ not Hopf Galois}

We consider now the degree 6 extensions $K/k$ which are not Hopf
Galois and determine in each case the minimal degree $[F:k]$ for a
field $F$ with $K\subseteq F\subseteq \wk$ such that the extension
$F/k$ is Hopf Galois.

\subsubsection{Case $|\Gal(\wk/k)|=24$}

We have 3 different conjugacy classes for $G$ inside $S_6$:
$2A_4$, $S_4(6d)$ and $S_4(6c)$. We consider the subgroup $G'=\Gal(\wk/K)$, which has order 4, and we ask whether there is a subgroup $C\subset G'$ of order 2 such that the fixed field $F=\wk^C$ provides a Hopf Galois extension $F/k$. Since the degree of this extension is 12, we deal again with the five isomorphism classes of subgroups of order 12 and its holomorphs. Since they have order divisible by 24, none of them
is excluded a priori.

The three groups $G$ under consideration have subgroups of order 12, therefore our first check will be for
 almost classically Galois extensions $F/k$.

The group $2A_4$ is isomorphic to $A_4\times C_2$ and an stabilizer $G'$
is isomorphic to $\langle \bigl( (12)(34), s\bigr)\rangle\times \langle \bigl( (13)(24), s\bigr)\rangle$, where
$s$ is a generator of $C_2$. It has three subgroups of order 2, two of them with normal complement $N\simeq  A_4\times 1\simeq A_4$.

\begin{prop}
Let $K/k$ be a separable extension of degree $6$ and let $\wk/k$ be its Galois closure.
If $\Gal(\wk/k)=2A_4$, there exist two fields $F$, with $K\subseteq F\subseteq \wk$,
such that $[F:k]=12$ and $F/k$ is almost classically Galois.
\end{prop}

When a degree 6 extension $K/k$ has Galois group isomorphic to
$S_4$, the subgroup $G'=\Gal(\wk/K)$ has order 4 and may be
isomorphic to  $C_2\times C_2$ or $C_4$. By considering the action
of $G$ on the cosets $G/G'$ we get $S_4(6d)$ in the first case and
$S_4(6c)$ in the second one.  In the first case,  $G'\subset S_4$
is generated by two disjoint transpositions and, therefore, it has
two subgroups having normal complement isomorphic to $A_4$. In the
second case, $G'\subset S_4$ is generated by a $4$-cycle $c$ and
its unique subgroup $C$ of order 2 is generated by $c^2$, a
product of two disjoint transpositions. The subgroup $C$ is then
contained in the alternating group $A_4$ and, therefore, has no
normal complement. We still study whether the intermediate field
$F=\wk^C$ can be a Hopf Galois extension. We deal once again with
the groups of order 12 and its holomorphs.

The symmetric group $S_{12}$ has two conjugacy classes of transitive subgroups isomorphic to $S_4$. In one of them, the stabilizer of a point is a group of order 2 with conjugacy class of length 6 and in the other one is a group of order 2 with conjugacy class of length 3. In our situation $C=\langle c^2\rangle$ has 3 conjugates, corresponding to the three different products of disjoint transpositions. Therefore the action of  $G$ on cosets $G/C$ gives the transitive group of degree 12 named $12T9=S_4(12e)$ in \cite{CHM}. In this group the elements of order 4 have disjoint cycle
decomposition type $(2)(2)(4)(4)$, but none of the  holomorphs of the groups of order 12 has elements with this decomposition type.

\begin{prop}
Let $K/k$ be a separable extension of degree $6$ and let $\wk/k$ be its Galois closure.
We assume that $G=\Gal(\wk/k) \simeq S_4$.

If $\Gal(\wk/K)$ is isomorphic to the Klein group (namely
$G=S_4(6d)$), there are two subgroups of order 2 having normal
complement in $G$. Therefore, there are two almost classically
Galois extensions $F/k$ of degree 12, with $K\subset F\subset
\wk$.

If $\Gal(\wk/K)$ is cyclic (namely $G=S_4(6c)$), there is a unique  field $F$
with $K\subsetneq F\subsetneq\wk$ and $F/k$ is not Hopf Galois.
\end{prop}

\subsubsection{Case $|\Gal(\wk/k)|=36$.}

When $G=F_{36}$, an stabilizer $G'$ is isomorphic to $S_3$ and
included in the normal subgroup of $G$ of order 18. On the other
hand, $G$ has no normal subgroups of order 12. Therefore, none of
the intermediate extensions will be almost classically Galois.

Since
the symmetric group  $S_{12}$ has a unique conjugacy class of transitive subgroups isomorphic to $G$, the class $12T17$,
it is enough to consider the holomorphs of the groups of order 12 and look for subgroups isomorphic to $G$.
The cyclic group is excluded for cardinality reasons and the remaining holomorphs are $12T81$, $12T83$ and $12T127$.
We check that none of them has a transitive subgroup isomorphic to $G$. Therefore, the intermediate extensions of degree 12 are not
Hopf Galois extensions.

The symmetric group $S_{18}$ has also a unique conjugacy class of transitive subgroups isomorphic to $G$,
the class $18T10$. Using Magma, if we take $N=\mathtt{SmallGroup}(18,4)$,
$$
N=\langle u,v,w \mid u^2=v^3=w^3=1,\ uv=v^2u,\ uw=w^2u\rangle,
$$
then the subgroup of $\Hol(N)$ generated by
$$
\begin{array}{l}
\sigma_1=(1, 10)(2, 16, 3, 13)(4, 11, 7, 12)(5, 17, 9, 15)(6, 14, 8, 18)\\
\sigma_2=(2, 3)(4, 7)(5, 9)(6, 8)(11, 12)(13, 16)(14, 18)(15, 17)\\
\sigma_3=(1, 5, 9)(2, 6, 7)(3, 4, 8)(10, 18, 14)(11, 16, 15)(12, 17, 13)\\
\sigma_4=(1, 8, 6)(2, 9, 4)(3, 7, 5)(10, 15, 17)(11, 13, 18)(12, 14, 16)
\end{array}
$$
is isomorphic to $G$. We get a different Hopf Galois structure if  we take $N=\mathtt{SmallGroup}(18,5)$.

\begin{prop}
Let $K/k$ be a separable extension of degree $6$ and let $\wk/k$ be its Galois closure.
Assume that $\Gal(\wk/k)=F_{36}$ and  $K\subset F\subset \wk$.
If $[F:k]=12$, then $F/k$ is not Hopf Galois.
If $[F:k]=18$, then $F/k$ is Hopf Galois but not almost classically Galois.
\end{prop}

\subsubsection{Case $|\Gal(\wk/k)|=48$.}

Now $G=2S_4$. An stabilizer $G'$  is a dihedral group of order 8
and  may have  intermediate extensions $F$ with $[F:k]=12$ or
$[F:k]=24$.

As for the normal subgroups of $G$, there is one of order 12 and
three of order 24,  the first one being contained in the other
three. Let us denote by $N$ this normal subgroup of $G$ of order
12. We see that $N\cap G'$ is a group of order 2, contained in all
the subgroups of $G'$ of order 4. Therefore, none of the
intermediate extensions $F$ such that $[F:k]=12$ is almost
classically Galois. On the other hand, when we intersect the three
different normal subgroups of order 24 with $G'$, we get the three
different subgroups of $G'$ of order 4, one cyclic and two Klein
groups. This gives that any order 2 subgroup of $G'$ different
from $N\cap G'$ has a normal complement in $G$ (in fact, two
different normal complements). Namely, there exist intermediate
fields $K\subset F\subset \wk$ such that $[F:k]=24$ and  $F/k$ is
almost classically Galois (with two different Hopf Galois
structures). It only remains to check whether there are Hopf
Galois structures, not almost classically Galois, for the degree
12 extensions.


When we consider a cyclic subgroup of order 4 of $G'$, the representation of $G$ as a transitive group of $S_{12}$ is $12T24$.
We proceed as before with the holomorphs of the groups of order 12. The only case where we find transitive
subgroups isomorphic to $G$ is $\Hol(A_4)$. But the conjugacy class of these subgroups is $12T22$.

When we consider a subgroup $V$ of $G'$ isomorphic to the Klein group, a priori we can obtain the transitive groups $12T21$, $12T22$ or $12T23$, since they are isomorphic to $G$ and in these groups the stabilizer of an element is a Klein group. According to the previous comment, we can only have a Hopf Galois structure if the action of $G$ on $G/V$ identifies  $G$ with a subgroup of $S_{12}$ in  the conjugacy class $12T22$. Considering this action, $V$ is the stabilizer of the coset $\Id V$.  The stabilizer of a point in $12T22$ is a Klein group formed by the identity, a permutation having exactly four fixed points and two permutations having exactly two fixed points.
Let us check the fixed points of the elements of $V$ acting on $G/V$.

The subgroup $V\subset G'\subset S_6$ can be $ \{\Id, (ab),\
(cd),\ (ab)(cd)\}$ or \linebreak $ \{\Id, (ab)(cd),\ (ac)(bd),\
(ad)(cb)\} $ with $(ab)$ and $(cd)$ disjoint transpositions. We
have that the unique transposition $(ef)$ which is disjoint with
both $(ab)$ and $(cd)$ belongs to $G$, and also $(ae)(bf)\in G$,
$(ce)(df)\in G$. In the first case, if $\tau_1=(ab)$ and
$\tau_2=(cd)$, we consider the cosets $C_1=\Id V$, $C_2=(ef)V$,
$C_3=(ae)(bf)V$ and $C_4=(ce)(df)V$. Then, $\tau_1(C_j)=C_j$ for
all $j\in\{1,2,4\}$ and $\tau_2(C_j)=C_j$  for all
$j\in\{1,2,3\}$. The second case is similar with $\tau_1=(ab)(cd)$
and $\tau_2=(ac)(bd)$. We consider $C_1=\Id V$, $C_2=(ef)V$,
$C_3=(ab)V=(cd)V$ and we see $\tau_i(C_j)=C_j$ for all
$i\in\{1,2\}$ and all $j\in\{1,2,3\}$.
 Therefore, in both cases the elements $\tau_i$ act on $G/V$ having at least 3 fixed points.
Namely, the elements $\tau_i\in V$ give rise to two permutations in $S_{12}$ having at least 3 fixed points.
Hence, the image of $V$ does not correspond to the case $12T22$.

\begin{prop}
Let $K/k$ be a separable extension of degree $6$ and let $\wk/k$ be its Galois closure.
Assume that $\Gal(\wk/k)=2S_4$ and  $K\subset F\subset \wk$.
If $[F:k]=12$, then $F/k$ is not Hopf Galois.

There exists an $F$ such that $[F:k]=24$ and $F/k$ is an almost classically Galois extension. More precisely,
let $C$ be the order 2 normal subgroup of $G'=\Gal(\wk/K)$. Then, any intermediate
extension $F$ with $[F:k]=24$ and $F\ne \wk^{C}$ is an almost classically Galois extension with at least two different
Hopf Galois structures.
\end{prop}

\subsubsection{Case $|\Gal(\wk/k)|=60$.}

When $G= A_5$, the group $G'$ is dihedral of order 10 and we can have intermediate extensions $F$ such that
$[F:k]$ is either 12 or 30. The extensions of degree 12 are not Hopf Galois, since none of the holomorphs of the groups of
order 12 has order divisible by 60. The extensions of degree 30 are not Hopf Galois either, since the four  groups of order 30
have solvable holomorph.

\begin{prop}
Let $K/k$ be a separable extension of degree $6$ and let $\wk/k$ be its Galois closure.
Assume that $\Gal(\wk/k)=A_5$ and  $K\subset F\subset \wk$.
Then, $F/k$ is not Hopf Galois.
\end{prop}

\subsubsection{Case $|\Gal(\wk/k)|=72$.}

Now we consider $G=F36:2$
and we have that $G'$ is a dihedral group of order 12 having  subgroups of order 2 generated
by transpositions.  They have normal complement in $G$,  since there is a subgroup of $G$ of
order 36 not containing transpositions.

The extensions with $[F:k]=24$ cannot be almost classically Galois, since $G$ has no normal subgroups of order 24.
They are not Hopf Galois extensions either: if  $C\subset G'$ is the unique subgroup of order 3, the action of $G$ on $G/C$
identifies $G$ with the conjugacy class $24T72$ and none of  the  holomorphs of the fifteen groups of order 24  has a
transitive subgroup isomorphic to $G$.


Now we consider a subgroup $V$ of $G'$ of order 4.  It is a Klein group $\langle (ab)\rangle\times \langle (cd)\rangle\subset S_6$
with $(ab)$ and $(cd)$ disjoint transpositions such that $(ab)(cd)$ belongs to the normal subgroup of $G$ of order 18. Therefore,
the extensions with $[F:k]=18$ are not almost classically Galois.

Through the action on $G/V$ the group $G$ is identified with the
subgroup $18T34$ of $S_{18}$. This is the unique conjugacy class
of transitive groups of degree 18  isomorphic to $G$ and having
non cyclic stabilizers. We look now at the holomorphs of the five
groups of order 18. When we consider the generalized dihedral
group $N=(C_3\times C_3)\rtimes C_2$, we find that the subgroup
$\langle \sigma_1, \sigma_2,\sigma_3,\sigma_4,\sigma_5\rangle$ of
$\Hol(N)$, where
$$
\begin{array}{rcl}
\sigma_1&=&(4, 7)(5, 8)(6, 9)(13, 16)(14, 17)(15, 18)\\
\sigma_2&=&            (1, 10)(2, 16, 3, 13)(4, 11, 7, 12)(5, 17, 9, 15)(6, 14, 8, 18)\\
\sigma_3&=&            (2, 3)(4, 7)(5, 9)(6, 8)(11, 12)(13, 16)(14, 18)(15, 17)\\
\sigma_4&=&            (1, 5, 9)(2, 6, 7)(3, 4, 8)(10, 14, 18)(11, 15, 16)(12, 13, 17)\\
\sigma_5&=&            (1, 8, 6)(2, 9, 4)(3, 7, 5)(10, 17, 15)(11, 18, 13)(12, 16, 14),\\
\end{array}
$$
is isomorphic to $G$.

Finally, when we consider a subgroup of $G'$ of order 6,
we see that the corresponding extension $F/k$ is not Hopf Galois because the holomorphs of the subgroups of order 12  have no
transitive subgroups isomorphic to $G$.

\begin{prop}
Let $K/k$ be a separable extension of degree $6$ and let $\wk/k$ be its Galois closure.
Assume that $\Gal(\wk/k)=F_{36}:2$ and  $K\subset F\subset \wk$.
\begin{itemize}
\item If $[F:k]=12$, then $F/k$ is not Hopf Galois.
\item If $[F:k]=18$, then $F/k$ is Hopf Galois but not almost classically Galois.
\item If $[F:k]=24$, then $F/k$ is  not Hopf Galois.
\item There exist intermediate extensions with $[F:k]=36$ such that $[F:k]$ is almost classically Galois. More
explicitly, this is so if $\Gal(\wk/F)$ is generated by a transposition in $S_6$.
\end{itemize}
\end{prop}

\subsubsection{Case $|\Gal(\wk/k)|=120$.}

We consider now the case $G=S_5$. The group $G'$ is isomorphic to
the Frobenius group $F_5$ and also isomorphic to $\Hol(C_5)$, the
holomorph of the cyclic group of order 5. All its subgroups of
order 2 are subgroups of $A_5$ and, therefore, the intermediate
extensions with $[F:k]=60$ are not almost classically Galois.
Since $A_5$ is the unique normal subgroup of $S_5$, no
intermediate extension can be almost classically Galois.

Let us assume that $F$ is an intermediate extension such that
$[F:k]=12,\ 24$ or $30$. All the groups of order 12 or 30 have a
solvable holomorph. The groups of order 24 have a solvable
holomorph except one, which has non solvable holomorph of order
8064. In any case, $G$ is not a subgroup of any of these
holomorphs and the corresponding extensions $F/k$ are not Hopf
Galois.

In order to check extensions with $[F:k]=60$ we must be aware that transitive groups of degree 60 exceeds the database's limit of Magma system, which is 32. But we can still work with the thirteen isomorphism classes of subgroups of order 60 and its holomorphs. We check all the 13 possibilities and there only remains the obvious candidate,
$\Hol(A_5)\simeq A_5\rtimes \Aut(A_5) \simeq A_5\rtimes S_5$: it is the only one with transitive subgroups isomorphic to $S_5$. But we should still check the compatibility with the actions on cosets. We compute these transitive subgroups of $\Hol(A_5)$ and we find two of them, which are conjugate in $S_{60}$.  We see that the stabilizer of an element has order 2 and is generated by an element with 6 fixed points. On the other hand, if $G''$ is the subgroup of order 2 of $G'$, then $G''=<\sigma>$ with
$C=\mbox{Centralizer}(G,\sigma)$ of order 8. The elements in $C$ give the fixed points of the action of $G''$ on cosets $G/G''$ since
$$
g\in C \iff g\sigma=\sigma g \iff  \sigma \cdot gG''=gG''.\\
$$
Taking into account that $g\in C \Rightarrow g\sigma \in C$ and $gG''=g\sigma G''$, we see
that there are only 4 fixed points. Therefore, the action of $G$ on cosets $G/G''$ does not identify $G$ with any of the subgroups of
$\Hol(A_5)$.

\begin{prop}
Let $K/k$ be a separable extension of degree $6$ and let $\wk/k$ be its Galois closure.
Assume that $\Gal(\wk/k)=S_5$ and  $K\subset F\subset \wk$.
Then, $F/k$ is not Hopf Galois.
\end{prop}

\subsubsection{Case $|\Gal(\wk/k)|=360$.}

We consider now the case $\Gal(\wk/k)\simeq A_6$. It is clear
that no intermediate extension $K\subset F\subset \wk$ can be
almost classically Galois. Let us see that the simplicity of $A_6$
allows also to discard the possibility of being a subgroup of the
corresponding holomorph.

The group $G'$ is isomorphic to $A_5$ and we can have intermediate extensions with $[F:k]=30,\ 36,\ 60,\ 72,\ 90,\ 120$ or $180$.
All the groups of order 30, 36, 90 or 180 have solvable holomorph. For the groups of order 60 there is only one case of non solvable
holomorph: $\Hol(A_5)$. But this holomorph has not simple subgroups of order 360.  When the order is 72 we have something similar: only one case of non solvable holomorph but with simple subgroups of order 168. Finally, with order 120 we have three cases of non solvable holomorph but their
simple subgroups are isomorphic to $A_5$.

\begin{prop}
Let $K/k$ be a separable extension of degree $6$ and let $\wk/k$ be its Galois closure.
Assume that $\Gal(\wk/k)=A6$ and  $K\subset F\subset \wk$.
Then, $F/k$ is not Hopf Galois.
\end{prop}

\subsubsection{Case $|\Gal(\wk/k)|=720$.}

In this last case $G=\Gal(\wk/k)=S_6$ and $G'\simeq S_5$. We have
subgroups of $G'$ of order 2 generated by transpositions, namely
having normal complement $A_6$ in $G$. The corresponding
intermediate extensions are almost classically Galois extensions.

For intermediate extensions with $[F:k]=180,\ 120,\ 90,\  72,\
60,\ 36,\ 30$ or $12$, we have already mentioned that there are no
regular groups with holomorph containing $G$, because the
holomorph is either solvable or does not have $A_6$ as a simple
subgroup.

The remaining possibilities are $[F:k]=144$ or $240$. Among the 197 isomorphism classes of groups of order 144 we find two cases of
holomorph having order divisible by 720 and having simple subgroups of order 360: $\mathtt{SmallGroup}(144,113)$ and
$\mathtt{SmallGroup}(144,197)$ according to Magma notation.
Among the 208 isomorphism classes of groups of order 240, there is also one in this situation: $\mathtt{SmallGroup}(240,208)$.
But none of these cases provide a holomorph with transitive subgroup isomorphic to $G$.

\begin{prop}
Let $K/k$ be a separable extension of degree $6$ and let $\wk/k$ be its Galois closure.
Assume that $\Gal(\wk/k)=S_6$ and  $K\subset F\subset \wk$.
\begin{itemize}
\item If $[F:k]=360$ and $\Gal(\wk/F)$ is generated by a transposition, then $F/k$ is almost classically Galois.
\item If $[F:k]<360$, then $F/k$ is not Hopf Galois.
\end{itemize}
\end{prop}

\section{Hopf Galois in prime degree}

 For $K/k$ a separable field extension of prime degree, Childs
\cite{Childs1} shows that
$$
K/k \mbox{ is Hopf Galois } \iff \Gal(\wk/k) \mbox{ is solvable}.
$$
Moreover, in this case $K/k$ is almost classically Galois and has
a unique Hopf Galois structure. More precisely, for $[K:k]=p$ with
$p=7$ or $11$, the solvable groups $\Gal(\wk/k)$ are $C_p$,
$D_{2\cdot p}$, and the Frobenius groups of orders $p(p-1)/2$ and
$p(p-1)$. In the cyclic case, $K|k$ is Galois; in the dihedral
case, $G'$ is a subgroup of order 2 and the cyclic subgroup of
order $p$ is a normal complement in $G$; in the Frobenius case,
$G'$ is the Frobenius complement and the Frobenius kernel is a
normal complement in $G$.

\section{A transitivity result}

In the preceding examples, we observed that if the extension $K/k$
is Hopf Galois, then $F|k$ is Hopf Galois for all $F$ with
$K\subset F \subset \wk$. This is due to the following result.

\begin{thm}\label{TEI}
 Let $K/k$ be a separable field extension and $\widetilde{K}/k$ its Galois closure. Let $F$ be a
 field with
  $K \subset F \subset \widetilde{K}$. If $K/k$ and $F/K$ are Hopf Galois extensions, then
   $F/k$ is also a Hopf Galois extension.
\end{thm}

\begin{proof} Let us denote $n=[K:k]$, $r=[F:K]$, $G=\Gal(\widetilde{K}/k)$, $G'=\Gal(\widetilde{K}/K)$,
and $G''=\Gal(\widetilde{K}/F)$. The action of $G'$ on left cosets
$G'/G''$ gives rise to a morphism $\psi: G'\to S_r$. Let $H$ be
its kernel and $\widetilde F=\wk^H$. The subgroup $H$ of $G'$ is
the intersection of all the stabilizers. It is maximal among the
normal subgroups of $G'$ contained in $G''$ and  $\widetilde F/K$
is the Galois closure of  $F/K$. Therefore, $\Gal(\widetilde
F/K)\simeq G'/H$ and $\Gal(\widetilde F/F)\simeq G''/H$.

Let  $y_1, \dots, y_r$ be a left transversal for  $G'/G''$.  The action of $g'\in G'$ is given by $g'\cdot y_jG''= y_{\psi(g')(j)} G''$.
On the other hand,  $y_1H,\dots,y_rH$  is a  transversal
for left cosets  $(G'/H)/(G''/H)$ and the action of $G'/H$ on $(G'/H)/(G''/H)$ provides a monomorphism $G'/H\hookrightarrow S_r$ which is the factorization of $\psi$ through its kernel:
$$\begin{array}{rcccc}
\psi: G' &\twoheadrightarrow & G'/H & \hookrightarrow & S_r \\
g'&\mapsto & g'H & \mapsto & \psi(g').\end{array}
$$

Let $x_1,\dots,x_n$ be a left transversal for the cosets $ G/G'$. The translation action of $G$ on $G/G'$ gives a monomorphism
$$\begin{array}{rcc} \varphi: G & \hookrightarrow & S_n \\ g & \mapsto & \varphi(g)\end{array}$$
defined by $gx_i \in x_{\varphi(g)(i)} G'$.

If the extension $K/k$ is Hopf Galois, there exists a regular subgroup  $N$ of  $S_n$ normalized by $G$.
If $F/K$ is Hopf Galois,  there exists a regular subgroup $R$ of  $S_r$ normalized by  $G'/H$ and, therefore, by
$\psi(G')$. That is, for every $a\in N,\  g \in G$, there exists $a'\in N$ such that $\varphi(g)a=a'\varphi(g)$, and for every
$b\in R,\  g' \in G'$, there exists $b'\in R$ such that $\psi(g')b=b'\psi(g')$.

Now we can take $x_iy_j, 1\leq i \leq n, 1 \leq j \leq r$ as left transversal for $G/G''$ and
$$gx_iy_j=x_{\varphi(g)(i)} g'y_j = x_{\varphi(g)(i)} y_{\psi(g')(j)} g''$$
gives the action of  $G$ on left cosets $G/G''$, which corresponds to a monomorphism
$$G \hookrightarrow S_{nr}=\Sym(\{1,\dots,n\}\times\{1,\dots,r\}).$$

We recall that we have an injective morphism $S_n\times S_r \hookrightarrow S_{nr}$. The element
 $(\sigma,\tau)$ in $S_n\times S_r$ corresponds to the permutation of $S_{nr}$ given by
$$\begin{array}{lccc} (\sigma,\tau): &\{ 1,\dots,n \}\times \{ 1,\dots,r\} & \rightarrow & \{ 1,\dots,n \}\times \{ 1,\dots,r\} \\ & (i_1,i_2) & \mapsto & (\sigma(i_1),\tau(i_2))\end{array}$$
If $N$ is a regular subgroup of $S_n$ and $R$ is a regular subgroup of $S_r$, then under the above monomorphism $N\times R$ is a regular subgroup of   $S_{nr}$: it is transitive and its order is $nr$.
In order to prove that $F/k$ is a Hopf Galois extension it is enough to check that $G$ normalizes $N\times R$:
let  $g \in G, a \in N, b \in R$. We have
$$\begin{array}{lll}
(g(a,b))(i_1,i_2)&=& g(a(i_1),b(i_2))=(\varphi(g)(a(i_1)),\psi(g')(b(i_2))) \\
&=& (a'\varphi(g)(i_1),b'\psi(g')(i_2))=((a',b')g )(i_1,i_2)
\end{array}$$
for all $(i_1,i_2)\in\{1,2,\dots, n\}\times \{1,2,\dots, r\}$.
\end{proof}

\begin{rem} {\rm The following example shows that the condition ``$F/K$ Hopf Galois'' in theorem \ref{TEI} is not superfluous.

Consider the alternating group $A_5$ and its holomorph
$$H=A_5\rtimes \Aut(A_5)=A_5\rtimes S_5.$$ The subgroup $G=A_5\rtimes A_5$ of $H$ is a transitive subgroup of $S_{60}$.

Let $\wk/\Q$ be an extension with Galois group $S_{60}$ and $k=\wk^G$. We denote $G'$ the
stabilizer in $G$ of a chosen element $i\in\{1,2,\dots, 60\}$ and $K=\wk^{G'}$.
Therefore $K/k$ is a separable field extension of degree 60 with Galois group $\Gal(\wk/k)=G$.
Since we have taken $G\subset \Hol(A_5)$, the extension $K/k$ is Hopf Galois.

On the other hand, $G'\simeq A_5$.  Let $G''$  be a subgroup of
$G'$ of order 12 and $F=\wk^{G''}$. Then, $F/K$ is a separable
degree 5 extension having normal closure $\wk/K$ with Galois group
isomorphic to $A_5$. Therefore, $F/K$ is not Hopf Galois.}
\end{rem}

\begin{rem} {\rm The following example shows that the condition
``$F\subset \wk$'' in Theorem \ref{TEI} is necessary.

Let $k=\Q$ and $K=\Q(\sqrt 5)$. The quadratic extension $K/\Q$ is Galois and therefore Hopf Galois. Now we take $F/K$ the cubic extension defined by a root of the irreducible  polynomial
$$Y^3-3(1+\sqrt 5)Y^2+\frac{9}2(5+3\sqrt 5)Y-\frac{27}2(1+\sqrt 5)\,\in K[Y].$$
Since all cubic separable extensions are Hopf Galois, so is $F/K$. The composition $F/\Q$ is an extension of degree 6.  A root of
the irreducible polynomial
$$X^6-6 X^5+9 X^4+243 X^3-729 X^2+1215 X-729\in \Q[X]$$
gives a primitive element for it. Since the Galois group of this
polynomial is the group $F_{36}$, the extension $F/\Q$ is not Hopf
Galois.}
\end{rem}

\vspace{1cm}
\footnotesize \noindent Teresa Crespo, Departament
d'\`Algebra i Geometria, Universitat de Barcelona, Gran Via de les
Corts Catalanes 585, E-08007 Barcelona, Spain, e-mail:
teresa.crespo@ub.edu

\vspace{0.3cm}
\noindent Anna Rio, Departament de Matem\`{a}tica Aplicada II, Universitat Polit\`{e}cnica de Catalunya, C/Jordi Girona, 1-3- Edifici Omega, E-08034 Barcelona, Spain, e-mail: ana.rio@upc.edu

\vspace{0.3cm}
\noindent Montserrat Vela, Departament de Matem\`{a}tica Aplicada II, Universitat Polit\`{e}cnica de Catalunya, C/Jordi Girona, 1-3- Edifici Omega, E-08034 Barcelona, Spain, e-mail:\linebreak montse.vela@upc.edu
\end{document}